\documentclass[11pt]{amsart}
\usepackage{amsfonts}

\usepackage[margin=1in]{geometry}
\usepackage{amsmath,amssymb,amsthm,mathtools}
\usepackage[colorlinks=true,linkcolor=blue,citecolor=blue,urlcolor=blue]{hyperref}

\numberwithin{equation}{section}
\newtheorem{theorem}{Theorem}[section]
\newtheorem{proposition}[theorem]{Proposition}
\newtheorem{lemma}[theorem]{Lemma}
\newtheorem{corollary}[theorem]{Corollary}

\newtheorem*{antonovthm}{Theorem A (Antonov)}

\newcommand{\et}[1]{\mathrm e\!\left(#1\right)}

\newcommand{\func}[1]{\operatorname{#1}}

\begin{document}
\title[Almost everywhere divergence]{Almost everywhere divergence of double Fourier series
	along shrinking conical regions}
\author{Ushangi Goginava}
\address{Department of Mathematical Sciences \\
	United Arab Emirates University, P.O. Box No. 15551\\
	Al Ain, Abu Dhabi, UAE}
\email{zazagoginava@gmail.com; ugoginava@uaeu.ac.ae}

\begin{abstract}
We study almost everywhere divergence of rectangular partial sums of double
trigonometric Fourier series along variable regions concentrated near the
diagonal. Fefferman's theorem shows that, in two variables, unrestricted
rectangular summation is radically different from the one-dimensional
Carleson-Hunt theory: there are continuous functions whose rectangular
Fourier sums diverge everywhere. Bakhvalov proved that this phenomenon
persists even when the indices are restricted to a fixed cone-shaped
neighbourhood of the diagonal. On the other hand, the diagonal summation
results of Tevzadze and Fefferman, and the later theorem of Antonov for
shrinking cones, show that convergence is restored when the aperture is of
order $1/n$.

We prove that Antonov's aperture condition is sharp in the $L^{2}$-scale.
For every positive nonincreasing sequence $\{\lambda _{n}\}$ with $%
\sup_{n}n\lambda _{n}=\infty $, we construct a function $f\in L^{2}(\mathbb{T%
}^{2})$ whose symmetric rectangular Fourier partial sums fail to converge
almost everywhere along the corresponding variable cone. The construction
combines a Fefferman-type analytic block with frequency separation and
independent random translations. We also observe that the sufficiency part
of Antonov's theorem does not require monotonicity of $\{\lambda _{n}\}$.
\end{abstract}

\maketitle

\section{Introduction}

Let 
\begin{equation*}
\mathbb{T}:=\mathbb{R}/\mathbb{Z}\simeq [0,1) 
\end{equation*}
be the one-dimensional torus with normalized Lebesgue measure, and write 
\begin{equation*}
\mathrm{e}(t):=\exp(2\pi i t). 
\end{equation*}
For $f\in L^1(\mathbb{T}^2)$ we use the Fourier coefficients 
\begin{equation*}
\widehat f(m,n):=\int_{\mathbb{T}^2} f(x,y)\mathrm{e}(-mx-ny)\,dx\,dy,
\qquad (m,n)\in\mathbb{Z}^2. 
\end{equation*}
The rectangular Fourier partial sums considered in this paper are the
symmetric sums 
\begin{equation}  \label{eq:symmetric-sums}
S_{M,L}(f;x,y):= \sum_{m=-M}^{M}\sum_{n=-L}^{L}\widehat f(m,n)\mathrm{e}%
(mx+ny), \qquad M,L\in\mathbb{N}_0 .
\end{equation}
For the proof of the analytic block estimate we also use the auxiliary
one-sided truncation 
\begin{equation}  \label{eq:auxiliary-R}
R_{M,L}(g;x,y):= \sum_{m=0}^{M-1}\sum_{n=0}^{L-1}\widehat g(m,n)\mathrm{e}%
(mx+ny), \qquad M,L\in\mathbb{N }.
\end{equation}
If $\widehat g(m,n)=0$ unless $m,n\ge0$, then 
\begin{equation}  \label{eq:S-R-relation}
S_{M-1,L-1}(g;x,y)=R_{M,L}(g;x,y), \qquad M,L\in\mathbb{N }.
\end{equation}
Thus the theorem is stated for the symmetric sums \eqref{eq:symmetric-sums},
while \eqref{eq:auxiliary-R} is used only for analytic polynomials with
nonnegative spectrum.

The pointwise convergence problem for trigonometric Fourier series goes back
to Luzin's problem. In one dimension, the examples of Kolmogorov and the
positive theorems of Carleson and Hunt established the central contrast
between divergence in $L^1$ and convergence in $L^p$, $p>1$ \cite%
{Carleson1966,Hunt1968,Kolmogoroff1923,Kolmogoroff1926,Luzin1951}. In
several variables the summation method becomes decisive. Diagonal, or cubic,
summation has positive almost everywhere convergence theorems due to
Tevzadze and Fefferman, with further refinements in Orlicz scales \cite%
{Fefferman1971Convergence,Sjolin1971,Tevzadze1970}. By contrast, Fefferman
proved that unrestricted rectangular summation admits continuous functions
of two variables whose rectangular Fourier sums diverge everywhere \cite%
{Fefferman1971Divergence}. Bakhvalov later showed that the same divergence
phenomenon persists even when the rectangular indices are confined to a
fixed cone-shaped neighbourhood of the diagonal \cite%
{Bakhvalov1997,Bakhvalov2002}. Related endpoint refinements, Orlicz-type
convergence theorems, and divergence constructions may be found in \cite%
{Antonov1996,Antonov2004,Antonov2009,Antonov2014,BakhbukhNikishin1973,
Konyagin1995,Konyagin2000,Sjolin1969,Stein1961,Zygmund1959}.

We shall use the following Orlicz notation. Put 
\begin{equation*}
\Phi_{2}(u):=u\bigl(\log(e+u)\bigr)^2  \log\log\log\bigl(e^{e^e}+u\bigr)%
,\qquad u\ge0. 
\end{equation*}
The Orlicz class $L^{\Phi_{2}}(\mathbb{T}^2)$, traditionally denoted by 
\begin{equation*}
L(\log^+L)^2(\log^+\log^+\log^+L)(\mathbb{T}^2), 
\end{equation*}
consists of all measurable functions $f$ for which 
\begin{equation*}
\int_{\mathbb{T}^2}\Phi_{2}(|f|/a)<\infty 
\end{equation*}
for some $a>0$. Since $\Phi_{2}(u)=o(u^2)$ as $u\to\infty$, one has 
\begin{equation*}
L^2(\mathbb{T}^2)\subset L^{\Phi_{2}}(\mathbb{T}^2). 
\end{equation*}

Let $\lambda=\{\lambda_n\}_{n\ge1}$ be a positive sequence. We define the
variable diagonal region 
\begin{equation}  \label{eq:omega-lambda}
\Omega_\lambda:= \left\{(M,L)\in\mathbb{N}^2: \frac{1}{1+\lambda_M}\le \frac{%
M}{L}\le 1+\lambda_L \right\}.
\end{equation}
For a constant sequence this is a fixed cone around the diagonal. If $%
\lambda_n\to0$, the sides of the cone approach the diagonal. The quantity $%
n\lambda_n$ is the effective width of the admissible strip at scale $n$: $%
\lambda_n=O(1/n)$ corresponds to bounded width, while $\sup_n
n\lambda_n=\infty$ allows arbitrarily wide windows at suitable scales.

Antonov proved that bounded width is sufficient for almost everywhere
convergence in the above Orlicz class. In the two-dimensional notation of
the present paper, the relevant consequence is the following.

\begin{antonovthm}
Let $\lambda=\{\lambda_n\}_{n\ge1}$ be positive and nonincreasing, and
assume that $\lambda_n=O(1/n)$. Then, for every $f\in L^{\Phi_{2}}(\mathbb{T}%
^2)$, the rectangular Fourier partial sums converge almost everywhere as $%
(M,L)\in\Omega_\lambda$ and $\min\{M,L\}\to\infty$ \cite[Theorem~2]%
{Antonov2017}.
\end{antonovthm}

The goal of this paper is to prove that Antonov's aperture condition is
necessary in the $L^2$-scale.

\begin{theorem}
\label{thm:main} Let $\lambda=\{\lambda_n\}_{n\ge1}$ be a positive
nonincreasing sequence such that 
\begin{equation*}
\sup_{n\ge1} n\lambda_n=\infty . 
\end{equation*}
Then there exists $f\in L^2(\mathbb{T}^2)$ such that the limit of $%
S_{M,L}(f;x,y)$, as $(M,L)\in\Omega_\lambda$ and $\min\{M,L\}\to\infty$,
does not exist for almost every $(x,y)\in\mathbb{T}^2$.
\end{theorem}

Combining Theorem~\ref{thm:main} with Antonov's theorem gives the following
sharp criterion in Antonov's Orlicz class, and hence in particular in the
$L^2$-scale, under the monotonicity assumption.

\begin{theorem}
\label{thm:characterization} Let $\lambda=\{\lambda_n\}_{n\ge1}$ be positive
and nonincreasing. The following assertions are equivalent:

\begin{enumerate}
\item for every $f\in L^{\Phi _{2}}(\mathbb{T}^{2})$, the sums $%
S_{M,L}(f;x,y)$ converge for almost every $(x,y)$ as $(M,L)\in \Omega
_{\lambda }$ and $\min \{M,L\}\rightarrow \infty $;

\item $\lambda _{n}=O(1/n)$.
\end{enumerate}
\end{theorem}

The proof of Theorem~\ref{thm:main} begins with a Fefferman-type analytic
polynomial whose mixed rectangular truncations are large on a set of
positive measure. These blocks are shifted to pairwise disjoint frequency
rectangles. The condition $\sup_n n\lambda_n=\infty$ is used to choose the
scales so that the active partial sums remain inside $\Omega_\lambda$.
Finally, independent random translations guarantee that almost every point
belongs to infinitely many translated large-value sets, while the resulting
block series still converges almost everywhere. Fixing one good random parameter gives
the required deterministic function.

\section{A Fefferman-type block polynomial}

For $N\in\mathbb{N}$ let 
\begin{equation}  \label{eq:FN}
F_N(s,t):=\frac1N\sum_{a=0}^{N-1}\sum_{b=0}^{N-1} \mathrm{e}%
\!\left(-ab/N\right)\mathrm{e}\!\left(as\right)\mathrm{e}\!\left(bt\right),
\qquad (s,t)\in\mathbb{T}^2.
\end{equation}

\begin{proposition}
\label{prop:block} For every $N\in \mathbb{N}$, 
\begin{equation*}
\left\Vert F_{N}\right\Vert _{L^{2}(\mathbb{T}^{2})}=1.
\end{equation*}%
Moreover, let 
\begin{equation*}
\Omega _{0}:=\left[ \frac{1}{4},\frac{1}{2}\right] \times \left[ \frac{1}{8},%
\frac{1}{4}\right] \subset \mathbb{T}^{2}.
\end{equation*}%
There exist absolute constants $c_{\ast},c_0>0$ and $N_0\in\mathbb{N}$ such
that, for all $N\ge N_0$, 
\begin{equation}
\left|S_{\lfloor Ny\rfloor-1,\lfloor Nx\rfloor-1}(F_N;x,y)\right| \ge
c_{\ast}\log N, \qquad (x,y)\in\Omega_0,  \label{eq:block-large-S}
\end{equation}
and consequently 
\begin{equation*}
\mu_2\left\{(x,y)\in\mathbb{T}^2: \left|S_{\lfloor Ny\rfloor-1,\lfloor
Nx\rfloor-1}(F_N;x,y)\right| \ge c_{\ast}\log N\right\} \ge c_0.
\end{equation*}
One may take $c_{\ast}=1/(4\pi)$ and $c_0=1/32$.
\end{proposition}

\begin{proof}
The Fourier coefficients of $F_N$ are 
\begin{equation*}
\widehat F_N(m,n)= 
\begin{cases}
N^{-1}\mathrm{e}\!\left(-mn/N\right), & 0\le m,n\le N-1, \\ 
0, & \text{otherwise.}%
\end{cases}%
\end{equation*}
Thus Parseval's identity gives 
\begin{equation*}
\left\|F_N\right\|_{L^2(\mathbb{T}^2)}^2 =\sum_{m,n\in\mathbb{Z}%
}\left|\widehat F_N(m,n)\right|^2
=\sum_{m=0}^{N-1}\sum_{n=0}^{N-1}\frac1{N^2}=1.
\end{equation*}

Let 
\begin{equation*}
M:=\left\lfloor Ny\right\rfloor ,\qquad L:=\left\lfloor Nx\right\rfloor
,\qquad \alpha :=Nx-L\in \lbrack 0,1),\qquad \beta :=Ny-M\in \lbrack 0,1).
\end{equation*}%
Then, by the definition of the auxiliary sums $R_{M,L}$, 
\begin{equation}
R_{M,L}(F_{N};x,y)=\frac{1}{N}\sum_{a=0}^{M-1}\sum_{b=0}^{L-1}\mathrm{e}%
\!\left( ax+by-ab/N\right) .  \label{eq:R-first}
\end{equation}%
Summing first in $b$ gives 
\begin{equation*}
R_{M,L}(F_{N};x,y)=\frac{1}{N}\sum_{a=0}^{M-1}\mathrm{e}\!\left( ax\right) 
\frac{1-\mathrm{e}\!\left( L(y-a/N)\right) }{1-\mathrm{e}\!\left(
y-a/N\right) }.
\end{equation*}%
Put $a=M-r$, $r=1,\dots ,M$. Since 
\begin{equation*}
y-\frac{M-r}{N}=\frac{\beta +r}{N},\qquad M=Ny-\beta ,\qquad L=Nx-\alpha ,
\end{equation*}%
we obtain the exact identity 
\begin{equation}
R_{M,L}(F_{N};x,y)=\frac{\mathrm{e}\!\left( Nxy\right) }{N}\sum_{r=1}^{M}%
\frac{\mathrm{e}\!\left( -(\beta +r)x\right) -\mathrm{e}\!\left( -\alpha
(\beta +r)/N\right) }{1-\mathrm{e}\!\left( (\beta +r)/N\right) }.
\label{eq:exact-identity}
\end{equation}

Fix $(x,y)\in\Omega_0$. Then $N/8-1\le M\le N/4$. For $1\le r\le M$ set 
\begin{equation*}
\theta_r:=\frac{\beta+r}{N}.
\end{equation*}
For $N\ge8$ we have $0<\theta_r\le3/8$. We next justify the two uniform
estimates that will be used in \eqref{eq:exact-identity}.

First, consider the reciprocal of $1-\mathrm{e}\!\left(\theta\right)$. As $%
\theta\to0$, 
\begin{equation*}
\mathrm{e}\!\left(\theta\right)=e^{2\pi i\theta} =1+2\pi
i\theta-2\pi^2\theta^2+O(\theta^3),
\end{equation*}
and therefore 
\begin{equation*}
1-\mathrm{e}\!\left(\theta\right) =-2\pi i\theta+2\pi^2\theta^2+O(\theta^3)
=-2\pi i\theta\left(1+i\pi\theta+O(\theta^2)\right).
\end{equation*}
Thus 
\begin{equation*}
\begin{aligned} \frac1{1-\et{\theta}} &=\frac1{-2\pi i\theta}
\left(1+i\pi\theta+O(\theta^2)\right)^{-1} \\ &=\frac{i}{2\pi\theta}
\left(1-i\pi\theta+O(\theta^2)\right) \\
&=\frac{i}{2\pi\theta}+\frac12+O(\theta). \end{aligned}
\end{equation*}
Hence the function 
\begin{equation*}
G(\theta):=\frac1{1-\mathrm{e}\!\left(\theta\right)}-\frac{i}{2\pi\theta}
\end{equation*}
satisfies 
\begin{equation*}
G(\theta)=\frac12+O(\theta) \qquad (\theta\to0).
\end{equation*}
The singular terms in the two summands defining $G$ have therefore
cancelled. Defining 
\begin{equation*}
G(0):=\frac12
\end{equation*}
makes $G$ continuous at $0$. On $(0,3/8]$ the function is continuous as
well, since $\mathrm{e}\!\left(\theta\right)\neq1$ there. Consequently $G$
is continuous on the compact interval $[0,3/8]$, and so there exists an
absolute constant $C_G$ such that 
\begin{equation*}
|G(\theta)|\le C_G, \qquad 0\le \theta\le \frac38.
\end{equation*}
It follows that, uniformly for all $r$ with $1\le r\le M$, 
\begin{equation*}
\begin{aligned} \frac1{N(1-\et{\theta_r})}
&=\frac1N\left(\frac{i}{2\pi\theta_r}+G(\theta_r)\right) \\ &=\frac{i}{2\pi
N\theta_r}+O\!\left(\frac1N\right) \\
&=\frac{i}{2\pi(\beta+r)}+O\!\left(\frac1N\right), \end{aligned}
\end{equation*}
because $N\theta_r=\beta+r$. The constant in this $O(1/N)$ term is
independent of $r$, $N$, $\alpha$, and $\beta$.

Second, since $0\leq \alpha <1$, Taylor's formula for the exponential gives 
\begin{equation*}
\mathrm{e}\!\left( -\alpha \theta _{r}\right) =e^{-2\pi i\alpha \theta
_{r}}=1+O(\alpha \theta _{r})=1+O\!\left( \frac{\beta +r}{N}%
\right) ,
\end{equation*}%
again uniformly in $r$. To combine these estimates, write 
\begin{equation*}
A_{r}:=\mathrm{e}\!\left( -(\beta +r)x\right) -1,\qquad \delta _{r}:=\mathrm{%
e}\!\left( -\alpha \theta _{r}\right) -1.
\end{equation*}%
Then $|A_{r}|\leq 2$ and $\delta _{r}=O(\theta _{r})=O((\beta +r)/N)$.
Therefore 
\begin{equation*}
\begin{aligned} &\frac{\et{-(\beta+r)x}-\et{-\alpha(\beta+r)/N}}
{N(1-\et{(\beta+r)/N})} \\ &\quad = (A_r-\delta_r)
\left(\frac{i}{2\pi(\beta+r)}+O\!\left(\frac1N\right)\right) \\ &\quad
=\frac{i}{2\pi}\frac{\et{-(\beta+r)x}-1}{\beta+r} +O\!\left(\frac1N\right).
\end{aligned}
\end{equation*}
Thus 
\begin{equation*}
\frac{\mathrm{e}\!\left( -(\beta +r)x\right) -\mathrm{e}\!\left( -\alpha
(\beta +r)/N\right) }{N(1-\mathrm{e}\!\left( (\beta +r)/N\right) )}=\frac{i}{%
2\pi }\frac{\mathrm{e}\!\left( -(\beta +r)x\right) -1}{\beta +r}+O\!\left( 
\frac{1}{N}\right)
\end{equation*}%
uniformly for $1\leq r\leq M$. Summing over $r$ and using $M\leq N/4$, %
\eqref{eq:exact-identity} yields 
\begin{equation}
R_{M,L}(F_{N};x,y)=\frac{i\mathrm{e}\!\left( Nxy\right) }{2\pi }%
\sum_{r=1}^{M}\frac{\mathrm{e}\!\left( -(\beta +r)x\right) -1}{\beta +r}%
+O(1),  \label{eq:R-asymptotic}
\end{equation}%
where the $O(1)$ term is absolute and uniform on $\Omega _{0}$.
Consequently, 
\begin{equation}
\begin{aligned} \left|R_{M,L}(F_N;x,y)\right| &\ge
\frac1{2\pi}\sum_{r=1}^{M}\frac1{\beta+r} -\frac1{2\pi}\left|\sum_{r=1}^{M}
\frac{\et{-(\beta+r)x}}{\beta+r}\right|-C_1. \end{aligned}
\label{eq:R-lower-start}
\end{equation}

We claim that 
\begin{equation}  \label{eq:osc-bound}
\left|\sum_{r=1}^{M}\frac{\mathrm{e}\!\left(-(\beta+r)x\right)}{\beta+r}%
\right| \le C_2, \qquad x\in[1/4,1/2],\quad \beta\in[0,1),\quad M\in\mathbb{N%
}.
\end{equation}
Indeed, 
\begin{equation*}
\sum_{r=1}^{M}\frac{\mathrm{e}\!\left(-(\beta+r)x\right)}{\beta+r} =\mathrm{e%
}\!\left(-\beta x\right)\sum_{r=1}^{M}\frac{\mathrm{e}\!\left(-rx\right)}{%
\beta+r}.
\end{equation*}
Let 
\begin{equation*}
A_m(x):=\sum_{r=1}^{m}\mathrm{e}\!\left(-rx\right).
\end{equation*}
Since $x\in[1/4,1/2]$, 
\begin{equation*}
\left|A_m(x)\right| =\left|\mathrm{e}\!\left(-x\right)\frac{1-\mathrm{e}%
\!\left(-mx\right)}{1-\mathrm{e}\!\left(-x\right)}\right| \le \frac2{\left|1-%
\mathrm{e}\!\left(-x\right)\right|} \le \frac2{2\sin(\pi/4)}=\sqrt2.
\end{equation*}
Abel summation gives 
\begin{equation*}
\sum_{r=1}^{M}\frac{\mathrm{e}\!\left(-rx\right)}{\beta+r} =\frac{A_M(x)}{%
\beta+M} +\sum_{r=1}^{M-1} A_r(x)
\left(\frac1{\beta+r}-\frac1{\beta+r+1}\right),
\end{equation*}
and hence 
\begin{equation*}
\left|\sum_{r=1}^{M}\frac{\mathrm{e}\!\left(-rx\right)}{\beta+r}\right| \le 
\frac{\sqrt2}{\beta+M} +\sqrt2\sum_{r=1}^{M-1}
\left(\frac1{\beta+r}-\frac1{\beta+r+1}\right) \le 2\sqrt2.
\end{equation*}
This proves \eqref{eq:osc-bound}.

On the other hand, 
\begin{equation*}
\sum_{r=1}^{M}\frac{1}{\beta +r}\geq \log N-(\log 8+1).
\end{equation*}%
Combining this with \eqref{eq:R-lower-start} and \eqref{eq:osc-bound}, we
get 
\begin{equation*}
\left\vert R_{M,L}(F_{N};x,y)\right\vert \geq \frac{1}{2\pi }\log
N-C_{3},\qquad (x,y)\in \Omega _{0},
\end{equation*}%
with an absolute constant $C_{3}$. Choosing $N_0$ sufficiently large gives 
\begin{equation*}
\left|R_{M,L}(F_N;x,y)\right| \ge \frac{1}{4\pi}\log N, \qquad
(x,y)\in\Omega_0, \qquad N\ge N_0.
\end{equation*}
Since $F_N$ has spectrum in the nonnegative quadrant, the relation %
\eqref{eq:S-R-relation} gives 
\begin{equation*}
S_{\lfloor Ny\rfloor-1,\lfloor Nx\rfloor-1}(F_N;x,y) =R_{\lfloor
Ny\rfloor,\lfloor Nx\rfloor}(F_N;x,y).
\end{equation*}
This proves \eqref{eq:block-large-S}. The measure estimate follows because $%
\Omega_0$ has measure $1/32$. Thus one may take $c_0=1/32$.
\end{proof}

\section{The shifted blocks}

Since $\sup_{n\geq 1}n\lambda _{n}=\infty $ then there exist strictly
increasing integers $\nu _{k}\rightarrow \infty $ such that 
\begin{align}
\nu _{k}& \geq 32\exp (2k),  \label{eq:nu-large} \\
\nu _{k}\lambda _{\nu _{k}}& \geq 32\exp (2k),  \label{eq:nulambda-large} \\
\nu _{k+1}& \geq 4\nu _{k}.  \label{eq:nu-separated}
\end{align}%
Define 
\begin{equation}
A_{k}:=\left\lfloor \nu _{k}/2\right\rfloor ,\qquad N_{k}:=\left\lfloor \min
\left\{ \frac{A_{k}}{8},\frac{A_{k}\lambda _{\nu _{k}}}{4}\right\}
\right\rfloor .  \label{eq:AkNk}
\end{equation}%
Then, after deleting at most finitely many terms, one has for every $k$, 
\begin{align}
N_{k}& \geq \exp (2k),  \label{eq:N-large} \\
A_{k}+N_{k}& <A_{k+1},  \label{eq:block-separation} \\
A_{k}+\frac{N_{k}}{2}& <\nu _{k},  \label{eq:below-nu} \\
\frac{N_{k}}{2A_{k}}& \leq \frac{\lambda _{\nu _{k}}}{8}.
\label{eq:lambda-small}
\end{align}%
Moreover, we may also assume $N_{k}\geq \max \{N_{0},16\}$, where $N_{0}$ is
the constant from Proposition \ref{prop:block}.

For $k\ge1$ set 
\begin{equation}  \label{eq:Pk}
P_k(x,y):=\mathrm{e}\!\left(A_kx+A_ky\right)F_{N_k}(x,y), \qquad (x,y)\in%
\mathbb{T}^2.
\end{equation}
Then 
\begin{equation*}
\left\|P_k\right\|_{L^2(\mathbb{T}^2)}=\left\|F_{N_k}\right\|_{L^2(\mathbb{T}%
^2)}=1.
\end{equation*}
The spectrum of $P_k$ is contained in 
\begin{equation}  \label{eq:Deltak}
\Delta_k:=\{A_k,A_k+1,\dots,A_k+N_k-1\}^2\subset\mathbb{Z}^2.
\end{equation}
By \eqref{eq:block-separation}, these spectral rectangles are pairwise
disjoint and ordered by increasing frequency.

For $u=(u_{1},u_{2})\in \Omega _{0}$ define the active symmetric radii 
\begin{equation}
M_{k}(u):=A_{k}+\left\lfloor N_{k}u_{2}\right\rfloor -1,\qquad
L_{k}(u):=A_{k}+\left\lfloor N_{k}u_{1}\right\rfloor -1.
\label{eq:active-indices}
\end{equation}

Since $u_1\in[1/4,1/2]$ and $u_2\in[1/8,1/4]$, we have 
\begin{align}
A_k+\frac{N_k}{8}-2 &\le M_k(u)\le A_k+\frac{N_k}{4}-1,  \label{eq:Mk-bounds}
\\
A_k+\frac{N_k}{4}-2 &\le L_k(u)\le A_k+\frac{N_k}{2}-1.  \label{eq:Lk-bounds}
\end{align}
In particular, because $u_2\le u_1$, one has $M_k(u)\le L_k(u)$.

\begin{lemma}
\label{lem:inside-region} For every $k$ and every $u\in\Omega_0$, 
\begin{equation*}
(M_k(u),L_k(u))\in\Omega_\lambda.
\end{equation*}
Moreover, the baseline square radii 
\begin{equation}  \label{eq:baseline}
(p_k,q_k):=(A_k-1,A_k-1)
\end{equation}
belong to $\Omega_\lambda$ for every $k$.
\end{lemma}

\begin{proof}
Since $M_k(u)\le L_k(u)$, 
\begin{equation*}
\frac{M_k(u)}{L_k(u)}\le1<1+\lambda_{L_k(u)}.
\end{equation*}
It remains to verify the lower bound in \eqref{eq:omega-lambda}. For $%
N_k\ge8 $, \eqref{eq:Mk-bounds} gives $M_k(u)\ge A_k$, while %
\eqref{eq:Lk-bounds} gives $L_k(u)\le A_k+N_k/2$. Hence 
\begin{equation*}
\frac{M_k(u)}{L_k(u)} \ge \frac{A_k}{A_k+N_k/2} =\frac{1}{1+N_k/(2A_k)}.
\end{equation*}
By \eqref{eq:below-nu}, $M_k(u)<\nu_k$. Since $\lambda_n$ is nonincreasing, 
\begin{equation*}
\lambda_{M_k(u)}\ge \lambda_{\nu_k}.
\end{equation*}
Using \eqref{eq:lambda-small}, 
\begin{equation*}
\frac{M_k(u)}{L_k(u)} \ge \frac{1}{1+\lambda_{\nu_k}/8} \ge \frac{1}{%
1+\lambda_{M_k(u)}}.
\end{equation*}
Thus $(M_k(u),L_k(u))\in\Omega_\lambda$.

Finally, $p_k=q_k=A_k-1$ and hence $p_k/q_k=1$. Therefore $%
(p_k,q_k)\in\Omega_\lambda$.
\end{proof}

\section{Random translations and the limiting function}

Let $\mu_2$ denote normalized Haar measure on $\mathbb{T}^2$. We set 
\begin{equation*}
\mathcal{X}:=(\mathbb{T}^2)^{\mathbb{N}},\qquad \mathcal{A}:=\mathcal{B}(%
\mathbb{T}^2)^{\otimes \mathbb{N}},\qquad \mathbb{P}:=\mu_2^{\otimes \mathbb{%
N}}.
\end{equation*}
Thus $(\mathcal{X},\mathcal{A},\mathbb{P})$ is the infinite product
probability space with product Haar measure. An element of $\mathcal{X}$ is
a sequence 
\begin{equation*}
\omega=(u_1,u_2,u_3,\ldots), \qquad u_k=(u_k^1,u_k^2)\in\mathbb{T}^2.
\end{equation*}
For each $k\ge1$ define the coordinate projection 
\begin{equation*}
\tau_k:\mathcal{X}\to\mathbb{T}^2, \qquad \tau_k(\omega)=u_k.
\end{equation*}
We write 
\begin{equation*}
\tau_k=(\tau_k^1,\tau_k^2), \qquad \tau_k^1(\omega)=u_k^1, \qquad
\tau_k^2(\omega)=u_k^2.
\end{equation*}
By the definition of the product measure, each coordinate projection has
distribution $\mu_2$. Thus, for every Borel set $B\subseteq\mathbb{T}^2$, 
\begin{equation*}
\mathbb{P}(\tau_k\in B)=\mu_2(B).
\end{equation*}
Moreover, the coordinate projections are independent. More precisely, let $%
k_1,\ldots,k_m$ be distinct and let $B_1,\ldots,B_m\subseteq\mathbb{T}^2$ be
Borel sets. Then 
\begin{equation*}
\mathbb{P}(\tau_{k_1}\in B_1,\ldots,\tau_{k_m}\in B_m) =\prod_{j=1}^m
\mu_2(B_j).
\end{equation*}
Thus the random variables $\tau_k$ are independent and Haar distributed on $%
\mathbb{T}^2$.

For $\omega\in\mathcal{X}$ define 
\begin{equation}  \label{eq:Hk}
H_k^\omega(x,y):=\frac1k P_k(x+\tau_k^1(\omega),y+\tau_k^2(\omega)).
\end{equation}
For every fixed $\omega$, the functions $H_k^\omega$ have pairwise disjoint
spectra, because translation only multiplies Fourier coefficients by
unimodular constants. Hence they are orthogonal in $L^2(\mathbb{T}^2)$, and 
\begin{equation*}
\left\|H_k^\omega\right\|_{L^2(\mathbb{T}^2)}=\frac1k.
\end{equation*}
Therefore 
\begin{equation*}
\sum_{k=1}^{\infty}\left\|H_k^\omega\right\|_{L^2(\mathbb{T}^2)}^2
=\sum_{k=1}^{\infty}\frac1{k^2}<\infty.
\end{equation*}
Consequently there exists $f^\omega\in L^2(\mathbb{T}^2)$ such that 
\begin{equation}  \label{eq:fomega}
f^\omega=\sum_{k=1}^{\infty}H_k^\omega \qquad\text{in }L^2(\mathbb{T}^2).
\end{equation}

\begin{lemma}
\label{lem:pointwise-random-conv} For every fixed $(x,y)\in\mathbb{T}^2$,
the series 
\begin{equation*}
\sum_{k=1}^{\infty}H_k^\omega(x,y)
\end{equation*}
converges for $\mathbb{P}$-almost every $\omega$. Consequently, 
\begin{equation*}
\mathbb{P}\left( \text{for a.e. }(x,y)\in\mathbb{T}^2,
\sum_{k=1}^{\infty}H_k^\omega(x,y)\text{ converges} \right)=1.
\end{equation*}
\end{lemma}

\begin{proof}
Fix $(x,y)\in\mathbb{T}^2$ and set 
\begin{equation*}
X_k(\omega):=H_k^\omega(x,y) =\frac1k
P_k(x+\tau_k^1(\omega),y+\tau_k^2(\omega)).
\end{equation*}
The random variable $X_k$ is a measurable function of the $k$th coordinate $%
\tau_k$. Since the coordinate projections $\tau_k$ are independent, the
random variables $X_k$ are independent.

Because $\tau _{k}$ is Haar distributed on $\mathbb{T}^{2}$ and Haar measure
is translation invariant, 
\begin{equation*}
\mathbb{E}X_{k}=\frac{1}{k}\int_{\mathbb{T}^{2}}P_{k}(x+u_{1},y+u_{2})%
\,du_{1}\,du_{2}=\frac{1}{k}\int_{\mathbb{T}^{2}}P_{k}(s,t)\,ds\,dt=0.
\end{equation*}%
Similarly, again by translation invariance, 
\begin{equation*}
\mathbb{E}\left\vert X_{k}\right\vert ^{2}=\frac{1}{k^{2}}\int_{\mathbb{T}%
^{2}}\left\vert P_{k}(x+u_{1},y+u_{2})\right\vert ^{2}\,du_{1}\,du_{2}=\frac{%
1}{k^{2}}\int_{\mathbb{T}^{2}}\left\vert P_{k}(s,t)\right\vert ^{2}\,ds\,dt.
\end{equation*}%
Since $\left\Vert P_{k}\right\Vert _{L^{2}(\mathbb{T}^{2})}=1$, it follows
that 
\begin{equation*}
\mathbb{E}\left\vert X_{k}\right\vert ^{2}=\frac{1}{k^{2}}.
\end{equation*}%
Consequently, 
\begin{equation*}
\sum_{k=1}^{\infty }\mathbb{E}\left\vert X_{k}\right\vert
^{2}=\sum_{k=1}^{\infty }\frac{1}{k^{2}}<\infty .
\end{equation*}

We now invoke Kolmogorov's convergence theorem for random series \cite[%
Theorem~2.5.6]{Durrett2019}: if $Y_k$ are independent real-valued random
variables with $\mathbb{E }Y_k=0$ and 
\begin{equation*}
\sum_{k=1}^{\infty}\func{Var}(Y_k)<\infty
\end{equation*}
then the series $\sum_k Y_k$ converges almost surely. Since the variables $%
X_k$ are complex-valued, we apply this theorem separately to their real and
imaginary parts. Write 
\begin{equation*}
X_k = U_k + iV_k, \qquad U_k := \func{Re} X_k, \qquad V_k := \func{Im} X_k.
\end{equation*}
The sequences $(U_k)$ and $(V_k)$ are independent sequences of real-valued
random variables, because they are obtained from the independent variables $%
X_k$ by measurable transformations. Moreover, 
\begin{equation*}
\mathbb{E }U_k=0, \qquad \mathbb{E }V_k=0.
\end{equation*}
Since 
\begin{equation*}
\left|X_k\right|^2=U_k^2+V_k^2,
\end{equation*}
we have $U_k^2\le \left|X_k\right|^2$ and $V_k^2\le \left|X_k\right|^2$. For
any real-valued square-integrable random variable $Y$, 
\begin{equation*}
\func{Var}(Y) = \mathbb{E}(Y-\mathbb{E}Y)^2 = \mathbb{E}Y^2-(\mathbb{E}Y)^2
\le \mathbb{E}Y^2.
\end{equation*}
Therefore 
\begin{equation*}
\func{Var}(\func{Re} X_k) = \func{Var}(U_k) \le \mathbb{E}U_k^2 \le \mathbb{E%
}|X_k|^2,
\end{equation*}
and similarly 
\begin{equation*}
\func{Var}(\func{Im} X_k) = \func{Var}(V_k) \le \mathbb{E}V_k^2 \le \mathbb{E%
}|X_k|^2.
\end{equation*}
Hence 
\begin{equation*}
\sum_{k=1}^{\infty}\func{Var}(\func{Re} X_k)<\infty, \qquad
\sum_{k=1}^{\infty}\func{Var}(\func{Im} X_k)<\infty.
\end{equation*}
Kolmogorov's convergence theorem applies to both real-valued series 
\begin{equation*}
\sum_{k=1}^{\infty}\func{Re} X_k \qquad\text{and}\qquad \sum_{k=1}^{\infty}%
\func{Im} X_k.
\end{equation*}
Thus the complex series $\sum_k X_k$ converges $\mathbb{P}$-almost surely.
Let 
\begin{equation*}
B(x,y,\omega):=\mathbf{1}_{\{\sum_k H_k^\omega(x,y)\text{ converges}\}}.
\end{equation*}
For each fixed $(x,y)$, 
\begin{equation*}
\int_{\mathcal{X}}B(x,y,\omega)\,d\mathbb{P}(\omega)=1.
\end{equation*}
Integrating over $\mathbb{T}^2$ and applying Fubini's theorem gives 
\begin{equation*}
\int_{\mathcal{X}}\mu_2\left( \left\{(x,y):\sum_k H_k^\omega(x,y)\text{
converges}\right\} \right)\,d\mathbb{P}(\omega)=1.
\end{equation*}
Since the integrand takes values in $[0,1]$, it must be equal to $1$ for $%
\mathbb{P}$-almost every $\omega$.
\end{proof}

For each $k$ define the translated bad set 
\begin{equation}  \label{eq:Ekomega}
E_k(\omega):=\{(x,y)\in\mathbb{T}^2:
(x+\tau_k^1(\omega),y+\tau_k^2(\omega))\in\Omega_0\}.
\end{equation}
Then 
\begin{equation*}
\mu_2(E_k(\omega))=\mu_2(\Omega_0)=\frac1{32}.
\end{equation*}

\begin{lemma}
\label{lem:limsup-full} One has 
\begin{equation*}
\mathbb{P}\left(\mu_2\left(\limsup_{k\to\infty}E_k(\omega)\right)=1\right)=1.
\end{equation*}
\end{lemma}

\begin{proof}
Fix $(x,y)\in\mathbb{T}^2$. The events 
\begin{equation*}
E_k(x,y):=\{\omega\in\mathcal{X}:(x,y)\in E_k(\omega)\}
\end{equation*}
are independent, because $E_k(x,y)$ depends only on $\tau_k$. Also, 
\begin{equation*}
\mathbb{P}(E_k(x,y))=\mu_2(\Omega_0)=\frac1{32}.
\end{equation*}
Hence 
\begin{equation*}
\sum_{k=1}^{\infty}\mathbb{P}(E_k(x,y))=\infty.
\end{equation*}
By the second Borel--Cantelli lemma, 
\begin{equation*}
\mathbb{P}\left((x,y)\in E_k(\omega)\text{ for infinitely many }k\right)=1
\end{equation*}
for every fixed $(x,y)$. Applying Fubini as in Lemma \ref%
{lem:pointwise-random-conv}, we conclude that 
\begin{equation*}
\mathbb{P}\left(\mu_2\left(\limsup_{k\to\infty}E_k(\omega)\right)=1\right)=1.
\end{equation*}
\end{proof}

\section{Proof of the main theorem}

\begin{proof}[Proof of Theorem \protect\ref{thm:main}]
Let 
\begin{equation*}
\mathcal{X}_{\mathrm{conv}}:=\left\{ \omega \in \mathcal{X}%
:\sum_{k=1}^{\infty }H_{k}^{\omega }(x,y)\text{ converges for a.e. }(x,y)\in 
\mathbb{T}^{2}\right\} ,
\end{equation*}%
and let 
\begin{equation*}
\mathcal{X}_{\mathrm{rec}}:=\left\{ \omega \in \mathcal{X}:\mu _{2}\left(
\limsup_{k\rightarrow \infty }E_{k}(\omega )\right) =1\right\} .
\end{equation*}%
By Lemmas \ref{lem:pointwise-random-conv} and \ref{lem:limsup-full}, 
\begin{equation*}
\mathbb{P}(\mathcal{X}_{\mathrm{conv}})=1,\qquad \mathbb{P}(\mathcal{X}_{%
\mathrm{rec}})=1.
\end{equation*}%
Therefore their intersection also has probability one. Indeed, 
\begin{equation*}
(\mathcal{X}_{\mathrm{conv}}\cap \mathcal{X}_{\mathrm{rec}})^{c}=\mathcal{X}%
_{\mathrm{conv}}^{c}\cup \mathcal{X}_{\mathrm{rec}}^{c},
\end{equation*}%
and so 
\begin{equation*}
\mathbb{P}\left( (\mathcal{X}_{\mathrm{conv}}\cap \mathcal{X}_{\mathrm{rec}})^{c}\right) \leq \mathbb{P}(\mathcal{X}_{\mathrm{conv}}^{c})+\mathbb{P}(%
\mathcal{X}_{\mathrm{rec}}^{c})=0.
\end{equation*}%
Hence 
\begin{equation*}
\mathbb{P}(\mathcal{X}_{\mathrm{conv}}\cap \mathcal{X}_{\mathrm{rec}})=1.
\end{equation*}%
In particular, this intersection is nonempty. Choose 
\begin{equation*}
\omega _{0}\in \mathcal{X}_{\mathrm{conv}}\cap \mathcal{X}_{\mathrm{rec}}
\end{equation*}%
and set 
\begin{equation*}
f:=f^{\omega _{0}}\in L^{2}(\mathbb{T}^{2}).
\end{equation*}

Let 
\begin{equation*}
E:=\limsup_{k\to\infty}E_k(\omega_0).
\end{equation*}
Since $\omega_0\in\mathcal{X}_{\mathrm{rec}}$, we have $\mu_2(E)=1$. Let 
\begin{equation*}
B:=\left\{(x,y)\in\mathbb{T}^2: \sum_{k=1}^{\infty}H_k^{\omega_0}(x,y)\text{
converges}\right\}.
\end{equation*}
Since $\omega_0\in\mathcal{X}_{\mathrm{conv}}$, one has $\mu_2(B)=1$. Hence 
\begin{equation*}
\mu_2(E\cap B)=1.
\end{equation*}

Fix $(x,y)\in E\cap B$. Because $(x,y)\in E$, there exist infinitely many $k$
such that 
\begin{equation*}
u_k(x,y):=(x+\tau_k^1(\omega_0),y+\tau_k^2(\omega_0)) \in\Omega_0.
\end{equation*}
For those $k$, write $u_k=(u_{k,1},u_{k,2})$ and set 
\begin{equation}  \label{eq:mkellk}
m_k:=M_k(u_k), \qquad \ell_k:=L_k(u_k).
\end{equation}
By Lemma \ref{lem:inside-region}, 
\begin{equation*}
(m_k,\ell_k)\in\Omega_\lambda, \qquad
(p_k,q_k):=(A_k-1,A_k-1)\in\Omega_\lambda.
\end{equation*}

Because the spectral rectangles $\Delta_j$ are pairwise disjoint and
ordered, the symmetric square partial sum at $(p_k,q_k)$ contains all
earlier blocks and no part of the $k$th block or any later block. Therefore 
\begin{equation}  \label{eq:baseline-sum}
S_{p_k,q_k}(f;x,y)=\sum_{j<k}H_j^{\omega_0}(x,y).
\end{equation}
Since $(x,y)\in B$, the series on the right converges as $k\to\infty$. Hence
the baseline subsequence $S_{p_k,q_k}(f;x,y)$ converges.

For the active indices $(m_k,\ell_k)$, spectral separation gives 
\begin{equation}  \label{eq:active-sum}
S_{m_k,\ell_k}(f;x,y) =\sum_{j<k}H_j^{\omega_0}(x,y)
+S_{m_k,\ell_k}(H_k^{\omega_0};x,y).
\end{equation}
Consequently, 
\begin{equation}  \label{eq:active-baseline-diff}
S_{m_k,\ell_k}(f;x,y)-S_{p_k,q_k}(f;x,y) =S_{m_k,\ell_k}(H_k^{\omega_0};x,y).
\end{equation}

By the definition of $H_{k}^{\omega _{0}}$, by the frequency shift %
\eqref{eq:Pk}, and by the corrected active radii \eqref{eq:active-indices}, 
\begin{equation*}
S_{m_{k},\ell _{k}}(H_{k}^{\omega _{0}};x,y)=\frac{\mathrm{e}\!\left(
A_{k}u_{k,1}+A_{k}u_{k,2}\right) }{k}S_{\left\lfloor
N_{k}u_{k,2}\right\rfloor -1,\left\lfloor N_{k}u_{k,1}\right\rfloor
-1}(F_{N_{k}};u_{k,1},u_{k,2}).
\end{equation*}%
Thus Proposition \ref{prop:block} yields 
\begin{equation}
\left\vert S_{m_{k},\ell _{k}}(H_{k}^{\omega _{0}};x,y)\right\vert \geq 
\frac{c_{\ast }\log N_{k}}{k}.  \label{eq:large-difference}
\end{equation}%
By \eqref{eq:N-large}, $N_{k}\geq \exp (2k)$, and therefore 
\begin{equation*}
\frac{c_{\ast }\log N_{k}}{k}\geq 2c_{\ast }.
\end{equation*}%
Combining this with \eqref{eq:active-baseline-diff}, we get 
\begin{equation}
\left\vert S_{m_{k},\ell _{k}}(f;x,y)-S_{p_{k},q_{k}}(f;x,y)\right\vert \geq
2c_{\ast }  \label{eq:gap}
\end{equation}%
for infinitely many $k$.

Let 
\begin{equation*}
L(x,y):=\sum_{j=1}^{\infty}H_j^{\omega_0}(x,y),
\end{equation*}
which exists because $(x,y)\in B$. The baseline subsequence $%
S_{p_k,q_k}(f;x,y)$ converges to $L(x,y)$. On the other hand, \eqref{eq:gap}
holds for infinitely many $k$. Therefore, for all sufficiently large such $k$%
, 
\begin{align*}
\left|S_{m_k,\ell_k}(f;x,y)-L(x,y)\right| &\ge
\left|S_{m_k,\ell_k}(f;x,y)-S_{p_k,q_k}(f;x,y)\right| \\
&\quad -\left|S_{p_k,q_k}(f;x,y)-L(x,y)\right| \\
&\ge c_{\ast}.
\end{align*}
Thus the active subsequence cannot converge to the same limit as the
baseline subsequence. Since both subsequences have indices in $%
\Omega_\lambda $ and tend to infinity, the symmetric rectangular partial
sums cannot converge along $\Omega_\lambda$ at $(x,y)$.

This holds for every $(x,y)\in E\cap B$, and $\mu_2(E\cap B)=1$. The proof
is complete.
\end{proof}

\section{The positive side without monotonicity}

\label{sec:positive-nonmonotone}

For completeness, we include the simple observation that Antonov's positive
result remains valid for arbitrary apertures satisfying $\lambda_n=O(1/n)$.
Only the negative theorem of the present paper uses monotonicity.

\begin{lemma}[Monotone majorization]
\label{lem:monotone-majorization} Let $\lambda=\{\lambda_n\}_{n\ge1}$ and $%
\mu=\{\mu_n\}_{n\ge1}$ be positive sequences such that $\lambda_n\le\mu_n$
for every $n$. Then 
\begin{equation*}
\Omega_\lambda\subseteq \Omega_\mu . 
\end{equation*}
\end{lemma}

\begin{proof}
If $(M,L)\in\Omega_\lambda$, then 
\begin{equation*}
\frac{1}{1+\lambda_M}\le \frac{M}{L}\le 1+\lambda_L . 
\end{equation*}
Since $\lambda_M\le\mu_M$ and $\lambda_L\le\mu_L$, we have 
\begin{equation*}
\frac{1}{1+\mu_M}\le \frac{1}{1+\lambda_M}\le \frac{M}{L}  \le
1+\lambda_L\le 1+\mu_L . 
\end{equation*}
Thus $(M,L)\in\Omega_\mu$.
\end{proof}

\begin{corollary}[Antonov's theorem for nonmonotone apertures]
\label{cor:antonov-nonmonotone} Let $\lambda=\{\lambda_n\}_{n\ge1}$ be an
arbitrary positive sequence such that $\lambda_n=O(1/n)$. Then for every $%
f\in L^{\Phi_{2}}(\mathbb{T}^2)$, and hence for every $f\in L^2(\mathbb{T}^2)
$, the sums $S_{M,L}(f;x,y)$ converge for almost every $(x,y)\in\mathbb{T}^2$
as $(M,L)\in\Omega_\lambda$ and $\min\{M,L\}\to\infty$.
\end{corollary}

\begin{proof}
Choose $C>0$ such that $\lambda_n\le C/n$ for every $n$, and set $\mu_n=C/n$%
. Then $\mu$ is positive, nonincreasing, and satisfies $\mu_n=O(1/n)$. By
Lemma~\ref{lem:monotone-majorization}, $\Omega_\lambda\subseteq\Omega_\mu$.
Antonov's theorem applied to $\mu$ gives almost everywhere convergence along
the larger region $\Omega_\mu$, and therefore along the smaller region $%
\Omega_\lambda$. The final statement for $L^2$ follows from $L^2(\mathbb{T}%
^2)\subset L^{\Phi_{2}}(\mathbb{T}^2)$.
\end{proof}

\section{Remarks and open questions}

\label{sec:remarks-open-questions}

We finish with two questions concerning the sharpness of Theorem~\ref%
{thm:main}.

\subsection*{The nonmonotone case}

Theorem~\ref{thm:main} assumes that $\{\lambda_n\}$ is nonincreasing. This
assumption is natural for the proof, because it allows one to pass from a
large value of $n\lambda_n$ to a whole interval of neighbouring indices
where the aperture remains sufficiently large. If $\lambda_n$ oscillates
strongly, the large values of $n\lambda_n$ may be isolated, and the present
block construction no longer guarantees that the active rectangular indices
remain inside $\Omega_\lambda$.

The positive direction is clear from Corollary~\ref{cor:antonov-nonmonotone}%
: if $\lambda_n=O(1/n)$, almost everywhere convergence holds without any
monotonicity assumption. The open problem is the converse.

\medskip \noindent\textbf{Question 1.} Let $\{\lambda_n\}$ be an arbitrary
positive sequence, not necessarily nonincreasing, and suppose that $\sup_n
n\lambda_n=\infty$. Does there always exist $f\in L^2(\mathbb{T}^2)$ whose
symmetric rectangular Fourier partial sums diverge almost everywhere along $%
\Omega_\lambda$?

\subsection*{Continuous counterexamples}

The function constructed here belongs to $L^2(\mathbb{T}^2)$, but the
construction does not give a continuous function. The reason is that the
argument uses an orthogonal series of frequency blocks: the coefficients
make the $L^2$-norms square summable, but they do not provide uniform
convergence. To obtain a continuous counterexample by the same method, one
would need much stronger uniform control of the blocks while preserving
their large rectangular partial sums. Balancing these two requirements is
the main obstruction.

For fixed cones, divergence phenomena for continuous functions are known
from the work of Bakhvalov; related results for summation methods between
cubic summation and fixed-cone summation were obtained by Antonov \cite%
{Antonov2014,Bakhvalov1997,Bakhvalov2002}. These results suggest that
continuous counterexamples may also exist for variable diagonal regions, but
they do not directly yield the statement below.

\medskip \noindent\textbf{Question 2.} Let $\{\lambda_n\}$ be positive and
nonincreasing, and suppose that $\sup_n n\lambda_n=\infty$. Does there exist
a continuous function $f\in C(\mathbb{T}^2)$ whose symmetric rectangular
Fourier partial sums diverge almost everywhere along $\Omega_\lambda$?

A positive answer would strengthen Theorem~\ref{thm:main} from the $L^2$
category to the continuous category. A negative answer would indicate that
variable diagonal summation has a genuine regularity threshold between $L^2(%
\mathbb{T}^2)$ and $C(\mathbb{T}^2)$.

\end{document}